\documentclass{amsart}

\usepackage{amsfonts,amssymb,amsmath,amsthm}

\usepackage{geometry}

\title{A note on purely infinite extensions}

\author{P. W. Ng}

\dedicatory{Dedicated to Professor Huaxin Lin on the occasion of his 70th birthday. }

\address{Department of Mathematics\\  University of Louisiana at Lafayette\\
217 Maxim Doucet Hall\\   1401 Johnston St.\\
Lafayette, Louisiana\\
70503\\  USA}

\email{png@louisiana.edu}

\author{Cangyuan Wang}

\address{School of Mathematical Sciences\\
	Ocean University of China\\
	238 Songling Road, Laoshan District\\
	Qingdao, Shandong\\
266100\\ 
P. R. China}

\email{cyw@ouc.edu.cn}

\newtheorem{thm}{Theorem}[section]  
\newtheorem{prop}[thm]{Proposition}

\newtheorem{lem}[thm]{Lemma}
\newtheorem{df}[thm]{Definition}

\theoremstyle{definition}
\newtheorem{para}[thm]{}

\newcommand{\A}{\mathcal{A}}
\newcommand{\B}{\mathcal{B}}
\newcommand{\C}{\mathcal{C}}
\newcommand{\D}{\mathcal{D}}
\newcommand{\E}{\mathcal{E}}

\newcommand{\K}{\mathcal{K}}
\newcommand{\J}{\mathcal{J}}

\newcommand{\oo}{\mathcal{O}}
\newcommand{\I}{\mathcal{I}}

\newcommand{\Mul}{\mathcal{M}}

\newcommand{\M}{\mathbb{M}}

\newcommand{\ee}{\mathbf{Ext}}    

\numberwithin{equation}{thm}

\begin{document}

\maketitle

\begin{abstract}
We study a strong pure infiniteness property for extensions, and give a classification result 
for extensions which contains many (and possibly all)
previous cases, under standard and mild assumptions. 

We also prove a Voiculescu-type theorem decomposition theorem for 
certain C*-subalgebras of a corona algebra, for which the 
inclusion map has this strong
pure infiniteness property.
\end{abstract}

\section{Introduction}

In their ground-breaking work, Brown, Douglas and Fillmore 
(BDF) classified all
essentially normal operators using K theory invariants. In the process,
they classified all essential C*-algebra extensions of the form 
$0 \rightarrow \K \rightarrow \E \rightarrow C(X) \rightarrow 0$, where
$\K$ is the C*-algebra of compact operators on a separable infinite dimensional
Hilbert space $l_2$, and $X \subset \mathbb{C}$ is a compact subset (see 
\cite{BDF1}). 
This phenomenal result kickstarted many important
developments in operator theory 
and noncommutative topology,
with the introduction
of many powerful tools like Kasparov's KK theory. 
(Some examples of the history and
 broader context can be found in
\cite{BDF1}, \cite{BDF2}, \cite{DouglasWeylLectures}, \cite{VoiculescuWvN}, \cite{ArvesonExt}, \cite{KasparovAbsorbing}, \cite{KasparovKK},   
\cite{RosenbergSchochet},
\cite{LinStableAUE}, \cite{HadwinMaShen}, etc.)

A central result in the BDF--Kasparov extension theory is the following theorem of Kasparov:  If $\A$ is a separable nuclear C*-algebra 
and $\B$ is a $\sigma$-unital C*-algebra, 
then there is a group isomorphism $KK^1(\A, \B) \cong Ext_a(\A, \B)$, where $Ext_a(\A, \B)$ is the group of unitary equivalence classes of absorbing extensions
$\A \rightarrow \C(\B \otimes \K)$ (\cite{KasparovKK} Section 7, Theorem 1).   This theory classifies absorbing extensions modulo the computability
of Kasparov's $KK^1$ (e.g., via the UCT \cite{RosenbergSchochet}), but it is incomplete in that the class of absorbing extensions can be a thin class,
and thus, many essential extensions need not be absorbing.  (Among other
things, every absorbing extension $\A \rightarrow \C(\B \otimes \K)$ is \emph{full}, i.e., brings nonzero elements of $\A$ to full
elements of $\C(\B \otimes \K)$.)  This led to the works of Lin and others in ``nonstable" extension theory, which focus on corona algebras 
with ``nice" properties (e.g., pure infiniteness, real rank zero), generalizing that the original BDF corona algebra is 
the Calkin algebra $\mathbb{B}(l_2)/\K$
which is simple purely infinite -- an important ingredient in the original BDF theory (e.g., see \cite{LinContScale}, \cite{LinVictoria}, \cite{LinExtII},
\cite{LinExt1997}, \cite{LinExtIII},  \cite{LinNgZ}, \cite{KNZPICorona},   \cite{NgPICorExt},  
\cite{NgExtFunctor2}, \cite{NgPIRR0}, \cite{LinNgSmallIdeals}, \cite{GabeLinNg}, \cite{NgWangExt1}).

In this short note, we move the purely infinite property from the corona algebra to the extension itself.  We obtain a (BDF-type) classification of such
purely infinite extensions, 
where the extensions need not be full, and the canonical ideal need not be stable (Theorem \ref{thm:MainThEnd}) -- and this  
classification seems to cover many (if not all) previous classified cases
(where the essential extension has large complement, and is of a 
$\sigma$-unital C*-algebra by a separable nuclear C*-algebra; and
the classification is up to (strong) unitary equivalence). 
The purely infinite property that we use is Gabe's interesting notion of
strong $O_{\infty}$-stability (\cite{GabeO2}; \cite{GabeOinfty} (1.1), Definition 4.1 and Remark 4.2), and we replace Kasparov's KK with Kirchberg's interesting ideal related KK theory
(\cite{KirchbergIdealKK}).  We also prove a  Voiculescu-type decomposition
theorem for completely positive maps from a strongly $O_{\infty}$-stable separable nuclear C*-subalgebra
of a corona algebra.

\section{A Voiculescu type decomposition theorem}

\begin{para} \label{FirstNotation} 
Let us begin by briefly recalling some notation. 
For a nonunital C*-algebra $\B$, we let $\Mul(\B)$ and $\C(\B) =_{df}
\Mul(\B)/\B$ denote the multiplier algebra and corona algebra of $\B$
respectively. Recall that for any C*-algebra $\A$, there is, up
to strong isomorphism, a
one-to-one correspondence between extensions of
$\B$ by $\A$ (i.e., extensions
$0 \rightarrow \B \rightarrow \E \rightarrow \A \rightarrow 0$) and
*-homomorphisms $\phi : A \rightarrow \C(\B)$ (the
\emph{Busby invariant} of the extension); e.g., see \cite{BlackadarKTh} 15.2, 15.3, 15.4.
Since we are only interested in properties of an extension that are invariant
under strong isomorphism,
following standard convention,
we identify a *-homomorphism $\A \rightarrow \C(\B)$
with the corresponding extension. 
Recall that an extension is essential
if and only if the corresponding Busby invariant is
injective.
Also, an extension $\A \rightarrow \C(\B)$ is said to be \emph{trivial} if it can be lifted to a *-homomorphism
$\A \rightarrow \Mul(\B)$.
Finally, we say that an extension  $\A \rightarrow \C(\B)$ has \emph{large complement}
if there is a projection $p \in \phi(\A)^{\perp}$ such that $p \sim 1_{\C(\B)}$. 
(In fact, in the definition of large complement, we can replace $\C(\B)$ with an arbitrary
unital C*-algebra $\C$.)
Some references for multiplier algebras, corona algebras and extension
theory, where further information can be found, are \cite{BlackadarKTh}, \cite{JensenThomsen}
and \cite{WeggeOlsen}.

For a C*-algebra $\D$ and for $a, b \in \D_+$, $a \preceq b$ (``$a$ is Cuntz subequivalent
to $b$")
means that there exists a sequence $\{ x_n \}$ in $\D$ such that
$x_n b x_n^* \rightarrow a$.  (This subequivalence generalizes Murray--von
Neumann subequivalence for projections.) 
For a subset $S
\subseteq \D$, we let $Ideal_{\D}(S)$ denote the ideal of $\D$
which is generated by $S$.  We often write $Ideal(S)$ in
place of $Ideal_{\D}(S)$.
Also, for any ideal $\J \subset \D$,
we let $\pi_{\J} : \D \rightarrow \D/\J$ denote the usual quotient map.
Again, we sometimes abbreviate by writing
$\pi : \D \rightarrow \D/\J$ instead of $\pi_{\J}$.

Finally, for any two elements $x, y$ in a C*-algebra and for any $\epsilon > 0$,
we let $x \approx_{\epsilon} y$ be a shorthand for $\| x - y \| < \epsilon$.
\end{para}

We introduce some pure infiniteness properties for extensions.

\begin{df} \label{df:PIExt}
Let $\A$ and $\B$ be C*-algebras with $\B$ nonunital,
and let $\phi : \A \rightarrow \C(\B)$ be an essential extension.
\begin{enumerate}
\item $\phi$ is \emph{purely infinite} if $\phi(a) \oplus \phi(a) \preceq \phi(a)$
in $\M_2(\C(\B))$, for all $a \in \A_+ - \{ 0 \}$.  Here, $\oplus$ means 
putting into the diagonal of $\M_2(\C(\B))$.
\item (\cite{GabeO2}; \cite{GabeOinfty}  (1.1), Definition 4.1 and Remark 4.2) $\phi$ is \emph{strongly $O_{\infty}$-stable} if
for $j = 1, 2$, there is a (norm-) continuous and (norm-) bounded path 
$\{ s^{(j)}_t \}_{t \in [0, \infty)}$      
in $\C(\B)$ such that for all $1 \leq j, k \leq 2$,  and for all $a \in \A$,
$$\lim_{t \rightarrow \infty} \| s^{(j)}_t \phi(a) - \phi(a) s^{(j)}_t \| = 0
\makebox{  and }  \lim_{t \rightarrow \infty} \| ((s^{(j)}_t)^* s^{(k)}_t - 
\delta_{j,k})\phi(a) \| = 0,$$
where $\delta_{j,k} = \begin{cases} 1 & j = k \\ 0 & j \neq k 
\end{cases}$.   
\end{enumerate}
\end{df}
\begin{para}  \label{para:PIDef}
If $\phi$ is strongly $O_{\infty}$-stable 
then $\phi$ 
is purely infinite:   Say that $\phi$ is strongly $O_{\infty}$-stable.    
For every integer $n \geq 1$,
let $v_n\in M_{2,1}(\C(\B)$) be defined as $v_n:=\begin{pmatrix}
(s^{(1)}_n)^*   \\
(s^{(2)}_n)^*
\end{pmatrix}$, where $\{ s^{(j)}_t \}_{t \in [0, \infty)}$ ($j = 1,2$)
is given
by Definition \ref{df:PIExt} (2).
Then from
 Definition \ref{df:PIExt} (2),  one can check that for all $a\in \A_+$, $\lim_{n \rightarrow \infty}\|v_n\phi(a)v^*_n-\phi(a)\oplus\phi(a)\|=0$. Hence, the Cuntz relation $\phi(a)\oplus\phi(a)\preceq\phi(a)$ holds for all 
$a \in \A_+$, and thus, $\phi$ is purely infinite.

Also, it is a short exercise to show than an extension $\phi : \A \rightarrow \C(\B)$ is full and purely infinite if and only if for any $a \in \A_+ - \{ 0 \}$,
there exists an $x \in \C(\B)$ for which $x \phi(a)x^* = 1_{\C(\B)}$.  
\end{para}

\begin{para}
Definitions \ref{df:PIExt} (1) and (2) roughly correspond to the pure infiniteness and $O_{\infty}$-stability
properties for C*-algebras, introduced in  the interesting work of Kirchberg and Rordam (e.g., \cite{KirchbergRordamPI},
\cite{KirchbergIdealKK}).  As far as we know, for an arbitrary separable nuclear C*-algebra, it is presently still an open problem 
whether weak pure infiniteness, pure infiniteness and strong pure infiniteness are equivalent properties (see \cite{KirchbergRordamPI}).
They are equivalent in the simple case.  Definition \ref{df:PIExt} (1) is also motivated by older works in extension theory;  for
example, when $\B$ is separable and
 stable, if one adds fullness of $\phi$ to the condition in Definition \ref{df:PIExt} (1), then the resulting
 condition is equivalent to the 
Elliott--Kucerovsky purely large condition (\cite{ElliottKucerovsky}). 
 \end{para}

\begin{para} \label{para:ForcedUnitization}
Let $\A$ be a C*-algebra. Recall that the \emph{forced unitization} of $\A$
is $$\A^+ =_{df} \begin{cases} \A^{\sim} & \makebox{ if $\A$ is nonunital,}\\
\A \oplus \mathbb{C} & \makebox{ if $\A$ is unital.}
\end{cases}$$  
Here, $\A^{\sim}$ is the usual minimal unitization of $\A$, when
$\A$ is nonunital.

Suppose that $\C$ is a unital C*-algebra and $\phi : \A \rightarrow \C$ is
a *-homomorphism. Then $\phi^+ : \A^+ \rightarrow \C$ is the *-homomorphism
given by $\phi^+ |_{\A} = \phi$ and $\phi^+(1) = 1_{\C}$.

Let $\B$ be a nonunital C*-algebra, and $\phi : \A \rightarrow \C(\B)$ be an extension with large
complement.  So $\phi^+$ is an extension. 
Then, if $\phi$ is essential then 
$\phi^+$ is essential.  And if $\phi$ is full and purely infinite, then
$\phi^+$ is full and purely infinite.    These statements are straightforward to prove, but let
us give the short proof for the latter statement:  Suppose that $\phi$ is full and purely
infinite.   Let $a \in \A^+$ be a nonzero positive element.  If 
	$a \in \A$,  then, since $\phi^+ |_{\A} = \phi$ and $\phi$ is full and
	purely infinite, we can find $x \in \C(\B)$ with $x \phi(a) x^* = 1_{\C(\B)}$  (see last paragraph of 
\ref{para:PIDef}); and so, $x \phi^+(a) x^*  = 1_{\C(\B)}$.
	Now suppose that $a = \alpha 1 + a_0$ where $\alpha > 0$ and $a_0 \in \A$.
	Then $\phi(a) = \alpha 1_{\C(\B)} + \phi(a_0)$. Since $\phi$ has   
	large complement, we can find a projection $p \in \phi(\A)^{\perp}$ such that
	$p \sim 1_{\C(\B)}$ (as a consequence, $p \perp \phi(a_0)$). Hence, let 
	$v \in \C(\B)$ be a partial isometry where $v^* v = p$ and $v v^* 
	= 1_{\C(\B)}$.  Hence, if we
	take $x =_{df} \frac{1}{\sqrt{\alpha}} v$, then
	$x \phi^+(a) x^* = 1_{\C(\B)}$.  Since $a \in (\A^+)_+ - \{ 0 \}$ was arbitrary,
	$\phi^+$ is full and purely infinite   (see last paragraph of \ref{para:PIDef}).  
\end{para}

Many examples of extensions (including absorbing extensions), which have been classified, are strongly
$O_{\infty}$-stable.

\begin{prop} \label{prop:Examples1}
Let $\A$ be a separable nuclear C*-algebra, $\B$ be a nonunital 
$\sigma$-unital C*-algebra, and let $\phi : \A \rightarrow \C(\B)$
be an essential extension with large complement.
Suppose that one of the following statements are true:
\begin{enumerate}
\item  $\B$ is stable and $\phi$ is absorbing.
\item  $\B$ is stable, $X$ is a topological space, $\B$ and $\A$ are $X$-C*-algebras, 
$\phi$ is weakly $X$-equivariant and $X$-absorbing -- i.e., $\phi$ absorbs all
weakly $X$-equivariant trivial extensions (\cite{GabeRuiz} Definitions 2.16 and 4.1, Theorem 6.11; \cite{GabeOinfty} Remark 11.2, Definition 11.13;
\cite{NgWangExt1} 4.14 last paragraph).
\item $\B$ is simple and has continuous scale.
\item $\C(\B)$ is properly infinite, and $\phi$ is full and purely infinite.
\item $\B$ is simple, pure, and has quasicontinuous scale. 
\item  $\B$ is $O_{\infty}$-stable.  
\end{enumerate}
Then  $\phi$ is strongly $O_{\infty}$-stable.
\end{prop}

\begin{proof}
Proof for Case (4):   If $\phi$ is full and purely infinite, then $\phi^+$ is full and purely infinite
(by \ref{para:ForcedUnitization}) and the result follows immediately from     
\cite{NgWangExt1} Proposition 3.8.  Cases (1) and (3) follow from case (4).
(See also \cite{NgWangVNAUniqueness} Lemma 3.2;  \cite{KucerovskyFredholm}
Lemma 1.3 and Theorem 1.4; the proof of the relevant direction does not need separability of 
$\B$, which can nonetheless be removed (from both directions) by standard reduction arguments; see also \cite{LinContScale}, \cite{LinSimpleCorona}.)

Case (5) follows from that if $\phi$ is essential then $\phi^+$ is essential
(see \ref{para:ForcedUnitization}), and from \cite{NgWangExt1} Proposition 4.9.  

Proof for Case (6):  By \cite{FarahSzabo} Theorem B, $\Mul(\B)$ is
separably $O_{\infty}$-stable, and hence, $\Mul(\B)$ and $\C(\B)$ are strongly
purely infinite (see \cite{KirchbergRordamPI} Definition 5.1 and Proposition 5.11). The result follows immediately from 
\cite{GabeOinfty} Theorem 9.7.  

Proof for Case (2):   We freely use the terminology of actions of topological spaces on C*-algebras without
explanation (e.g., see \cite{GabeOinfty}, \cite{NgWangExt1}, \cite{GabeRuiz}).   \emph{The reader might wish to skip this
proof on a first reading, and return to it only after completing the rest of the paper.}
Since $\phi$ is weakly $X$-equivariant, let $\phi_0 : \A \rightarrow \Mul(\B)$ be a completely positive linear 
map with $\pi \circ \phi_0 = \phi$ such that $\phi_0$ is weakly $X$-equivariant (\cite{GabeRuiz} Definitions 2.16 and 4.1; \cite{GabeOinfty}
Remark 11.2).  
By the Kasparov--Stinespring Theorem with operator convex cones 
 (\cite{GabeRuiz} Corollary 2.26),  let $\theta : \A \rightarrow \Mul(\B)$ be a weakly $X$-equivariant
*-homomorphism (i.e., a trivial extension) and $y \in \Mul(\B)$ such that $y \theta(\cdot) y^* = \phi_0(\cdot)$.  
(Note that we may need to initially 
replace $\phi_0$ with $\phi_0/C$ for some $C > 0$ to ensure that 
we have a contractive map to apply \cite{GabeRuiz} Cor. 2.26.)  We construct the infinite repeat (or generalized Lin extension) $\Theta$ as follows:
Since $\B$ is stable, let $\{ P_n \}$ be a sequence of pairwise orthogonal projections in $\Mul(\B)$ for which 
$P_n \sim 1_{\Mul(\B)}$ for all $n$, and $\sum_{n=1}^{\infty} P_n = 1_{\Mul(\B)}$, where the sum converges strictly in $\Mul(\B)$.
 Let $S \in \Mul(\B)$ be an isometry with $SS^* \leq P_1$ and $P_1 - SS^* \sim 1_{\Mul(\B)}$.  
Let $\theta_{1,1} : \A \rightarrow P_1\Mul(\B)P_1$ be the *-homomorphism, with large complement, given by 
$\theta_{1,1} =_{df} S \theta (\cdot) S^*$.   For all $n \geq 1$, let $e_{1, n} \in \Mul(\B)$ be a partial isometry where
$e_{1,1} = P_1$, $e_{1, n} e_{1, n}^* = e_{1,1}$ and $e_{1,n}^* e_{1,n} = P_n$.  For all $m, n \geq 1$, define
$e_{m,n} =_{df} e_{1,m}^* e_{1, n}$.   We define the generalized Lin extension $\Theta : \A \rightarrow \Mul(\B)$ by
$\Theta(\cdot) =_{df} \sum_{k=1}^{\infty} e_{k,1} \theta_{1,1}(\cdot) e_{1,k}$, where the series converges in the pointwise-strict topology.
Hence, $\Theta$ is a weakly $X$-equivariant trivial extension with large complement. Also, for all $m, n \geq 1$, for all $a \in \A$, 
$e_{m,n} \Theta(a) = e_{m,n} \sum_k e_{k,1} \theta_{1,1}(a) e_{1,k} = e_{m,n} e_{n,1} \theta_{1,1}(a) e_{1,n} 
= e_{m, 1} \theta_{1,1}(a) e_{1,n} = 
e_{m,1} \theta_{1,1}(a) e_{1,m} e_{m,n} = \left(\sum_k e_{k,1} \theta_{1,1}(a) e_{1,k}\right) e_{m,n} = \Theta(a) e_{m,n}$.
Hence, $\Theta(\A)'$ contains a unital copy of $\mathbb{B}(l_2)$, and hence contains a unital copy of $O_{\infty}$.  Hence,
$\Theta$ is strongly $O_{\infty}$-stable.  Hence, since $\Theta$ approximately dominates $\phi_0$ (see \cite{GabeOinfty} Definition 3.1), 
$\pi(\Theta) \oplus \phi$ is strongly $O_{\infty}$-stable (\cite{GabeOinfty} Proposition 4.10), where we take $\oplus$ to be the classical BDF sum (see,
e.g., \cite{GabeOinfty} Remark 5.1 where it is called the ``Cuntz sum";  see also \cite{ElliottKucerovsky} 4, where they take the quotient.)
Since $\phi$ absorbs every weakly $X$-equivariant trivial extension, $\pi(\Theta) \oplus \phi$ is unitarily equivalent to $\phi$ (with unitary 
from $\Mul(\B)$).  Hence, $\phi$ is strongly $O_{\infty}$-stable.
\end{proof}

\emph{The reader  might wish to 
skip \ref{para:IdealAbsorption}, \ref{para:Examples2} and \ref{para:Examples3} 
on a first reading, and to return
to them only after completing the rest of the paper.} 

\begin{para} \label{para:IdealAbsorption}
While the theory in Case (2) of Proposition \ref{prop:Examples1} is very interesting (e.g., see \cite{GabeOinfty}), for the purpose of classifying
extensions, it is not enough, since the relevant ideals are all strictly closed:  
Say that $\phi$ is a trivial essential extension that satisfies
the conditions of Case (2).  So let $\phi_0 : \A \rightarrow \Mul(\B)$ be a  weakly 
$X$-equivariant and $X$-absorbing *-homomorphism such that
$\pi \circ \phi_0 = \phi$.    Then by \cite{GabeRuiz} Lemma 6.6, 
for all $a \in \A$, $\overline{\Mul(\B) \phi(a) \Mul(\B)} = \overline{\Mul(\B) \phi(a) \Mul(\B)}^{strict} 
= \overline{\B \phi(a) \B}^{strict}$,  where $\overline{(\cdot)}$ is the norm-closure and $\overline{(\cdot)}^{strict}$ is the 
closure in the strict topology.  For example, if $\C$ is a unital simple AF algebra with infinitely many extreme tracial 
states, then $\Mul(\C \otimes \K)$ has infinitely many ideals, but none of them are strictly closed except for
$\{ 0 \}$ and $\Mul(\C \otimes \K)$. 
\end{para}

\begin{para} \label{para:Examples2}  
One can construct 
many examples of a strongly $O_{\infty}$-stable
essential extensions that do not fall within the cases mentioned in
Proposition \ref{prop:Examples1}.  Here are two quick examples: 
Let $\A$ be any unital separable nuclear C*-algebra which can be unitally
embedded into $O_{\infty}$ (e.g., $\A = C(\mathbb{T}^2)$, where 
$\mathbb{T}^2$ is the $2$-torus).  Let $\C$ be any unital simple AT-algebra
with $K_1(\C) \neq 0$ and having infinitely many extreme tracial states. 
Let $\B =_{df} \C \otimes \K$.  By \cite{KNZPICorona}, $\C(\B)$ is not 
purely infinite.
Let $\phi_0 : \A \rightarrow \Mul(\B)$ be any nonfull *-monomorphism
for which $\phi_0(\A) \cap \B = \{ 0 \}$.  Let $\phi =_{df} 
\pi \circ \iota \circ \phi_0 : \A \rightarrow \C(\B \otimes \K)$  
where $\pi : \Mul(\B \otimes \K) \rightarrow \C(\B \otimes \K)$ is the
usual quotient map, and $\iota : \Mul(\B) \rightarrow \Mul(\B) \otimes 1_{\Mul(\K)} \subset \Mul(\B) \otimes \Mul(\K) \subset \Mul(\B \otimes \K)$
 is the obvious *-embedding.
  Then $\phi$ is a (trivial) essential
extension which is strongly
$O_{\infty}$-stable -- in particular, if $S, T \in \Mul(\K) = \mathbb{B}(l_2)$
are isometries with $SS^* + TT^* = 1_{\Mul(\K)}$, 
then $\pi(1_{\Mul(\B)} \otimes S)$ and $\pi(1_{\Mul(\B)} \otimes T)$ witness
that each element of 
$\phi(\A)$ commutes with a copy of $O_2$ 
which is a unital C*-subalgebra
of $\C(\B \otimes \K)$. Since we have a unital *-embedding 
$O_{\infty} \subseteq O_2$,  $\phi$
is strongly $O_{\infty}$-stable. 
Since $\C(\B \otimes \K)$ is not purely infinite and since $\phi$ is nonfull,
$\phi$ does not fall within the cases in Proposition \ref{prop:Examples1}.    
 
Even though $\phi$ as above is a trivial extension, this is not a serious issue, since a trivial extension
need not be the zero element of the relevant extension group, outside of the
BDF--Kasparov theory.  Nonetheless, here is a quick nontrivial example:
Let $\A$, $\C$ and $\B$ be as above.
By \cite{LinContScale} Remark 2.9,
let $\J \subseteq \Mul(\B)$ be the minimal ideal of 
$\Mul(\B)$ which properly contains $\B$. 
By \cite{LinNgZ} Proposition 4.2 and \cite{ThielRanks} Corollary 4.6 (see also \cite{KNZPICorona} Theorems 5.7 and 5.8) and \cite{KNZMin}
Theorem 5.6, we can find a projection $P \in \J - \B$.
Then $\pi(P) \C(\B) \pi(P) = \C(P \B P)$ is
simple, and hence, simple purely infinite (see \cite{LinSimpleCorona}
Cor. 3.3); and thus, $P \B P$ has continuous scale.  
By \cite{LinNgSmallIdeals} Theorem 7.4, and since  
$K_1(\B) \neq 0$, we can find a nonzero projection $q \in \C(P \B  P)$
such that $q$ cannot be lifted to a projection in $\Mul(\B )$.
Now, $q \pi(\J)q$ contains a unital copy of $O_{\infty}$. 
Since $\A$ can be unitally embedded into $O_{\infty}$,
we certainly have a nontrivial essential extension 
$\psi : \A \rightarrow q \C(\B) q \subseteq \C(P\B P)
\subseteq \C(\B )$. 
Note that since $\C(P \B P)$ is simple purely infinite, $\pi(P) - q$ is properly
infinite.  Also, $1_{\Mul(\B)} - P \sim 1_{\Mul(\B)}$ (which is properly infinite) by
\cite{LinNgZ} Proposition 4.2 (see also \cite{KNZMin} Theorem 5.6).
 From the above, and 
by Proposition \ref{prop:Examples1} (2),
$\psi$ is strongly $O_{\infty}$-stable.    However, since $\C(\B)$ is not purely 
infinite and since $\psi$ is nonfull, $\psi$ is not contained in the cases in Proposition 
\ref{prop:Examples1} (and $\psi$ is not a trivial extension).
\end{para}

\begin{para}  \label{para:Examples3} Here is a quick example of a nonstrongly $O_{\infty}$-stable essential extension:
Let $\C$ be a simple unital AF-algebra with infinitely many extreme tracial states, and let 
$\B =_{df} \C \otimes \K$.  By \cite{KNZPICorona}, $\C(\B)$ is not purely infinite. By
\cite{LinWvN2}, $\C(\B)$ has real rank zero.  Hence, we can find a nonzero projection
$r \in \C(\B)$ for which $r \oplus r \not\preceq r$.   
Let $\phi : \mathbb{C} \rightarrow \C(\B)$ be the essential extension given by $\phi(1) =_{df} r$.
Then $\phi$ is an essential extension which is not strongly $O_{\infty}$-stable.
\end{para}

We now move towards a Voiculescu-type decomposition (Lemma \ref{lem:Approx1StepInner}, 
Lemma \ref{lem:Mar720261AM} and Theorem \ref{thm:NgWangWvN1}).   Part of the main  point
is that the strong $O_{\infty}$-stability condition simplifies the proofs, making them go very
smoothly.

\begin{lem} \label{lem:Approx1StepInner}
Let $\B$ be a nonunital C*-algebra, and 
let $\A \subset \C(\B)$ be a separable nuclear C*-subalgebra, such that
the inclusion map $\A \hookrightarrow \C(\B)$ is strongly $O_{\infty}$-stable.

Suppose that $\psi: \A \rightarrow \C(\B)$ is a completely positive contractive
linear map for which $\psi(a) \in Ideal_{\C(\B)}(a)$ for all $a \in \A$.

Then $\psi$ can be approximated, in the pointwise-norm topology, by 
one-step inner maps, i.e., maps
of the form $a \mapsto x a x^*$, where $x \in \C(\B)$.   

Moreover, for the above approximating one-step inner maps, we can
choose $x$ so that $\| x \| \leq 2$.
\end{lem}

\begin{proof}
 This follows immediately from \cite{GabeO2} Theorem 3.22.  

The essence of the argument is the Hahn--Banach separation type theorem in 
\cite{KirchbergRordamHomotopy} Corollary 4.3, together with computations using
the strong $O_{\infty}$-stability condition. (Actually, one only needs (not necessarily strong)
$O_{\infty}$-stability of the map, which gives a ``uniform proper infiniteness" on finite sets of 
positive elements.   See \cite{GabeOinfty} Remark 4.2.  Also compare with \cite{KirchbergRordamPI} Prop. 8.4 (iii).)

\end{proof}

\begin{lem} \label{lem:Mar720261AM}
Let $\B$ be a nonunital C*-algebra, and let $\A \subset C(\B)$ be 
a separable nuclear C*-subalgebra 
such that the inclusion map $\A \hookrightarrow \C(\B)$ is strongly 
$O_{\infty}$-stable. 

Suppose that $\psi : \A \rightarrow \C(\B)$ is a completely positive contractive
linear map for which $\psi(a) \in Ideal_{\C(\B)}(a)$ for all $a \in \A$.

Then we can find a (norm-) bounded
sequence $\{ x_n \}$ in $\C(\B)$ 
such that for all $a \in \A$,
$$x_n a x_{n+1}^* \rightarrow 0 \makebox{ and } x_n a x_n^* 
\rightarrow \psi(a), \makebox{ as } 
n \rightarrow \infty.$$                                         
\end{lem}

\begin{proof}
Since the inclusion map $\A \hookrightarrow \C(\B)$ is strongly $O_{\infty}$-stable,
let $\{ s^{(j)}_t \}_{t \in [0, \infty)}$ ($j=1,2$) be
\emph{bounded} norm continuous paths in $\C(\B)$ which realize this (as in Definition
\ref{df:PIExt} (2), with $\phi$ replaced with the inclusion map).

Choose a strictly
increasing sequence $\{ t_n \}_{n=1}^{\infty}$ in $[0, \infty)$ such that 
\begin{equation} \label{equ:Mar620268AM}
t_n \rightarrow \infty \makebox{ and for $j=1,2$, } 
\| s^{(j)}_{t_n} - s^{(j)}_{t_{n+1}} \| \rightarrow 0.
\end{equation}
Now by Lemma \ref{lem:Approx1StepInner}, 
let $\{ y_n \}$ be a sequence in $\C(\B)$, with $\| y_n \| \leq 2$ for all $n$,
such that for all $a \in \A$, 
\begin{equation} \label{equ:Mar620268PM}
y_n a y_n^* \rightarrow \psi(a).
\end{equation} 
For all $n \geq 1$, let $x_n \in \C(\B)$ be given by 
$$
x_n =_{df} \begin{cases}
y_n (s^{(1)}_{t_n})^* & \makebox{ if $n$ is odd}\\
y_n (s^{(2)}_{t_n})^* & \makebox{ if $n$ is even}\\
\end{cases}
$$ 

Let $\epsilon > 0$ and $a \in \A$ be arbitrary.   
For all \emph{sufficiently large $n$,} with $n$ being odd, we have that  
\begin{eqnarray*}
x_n a x_{n+1}^* & = & y_n (s^{(1)}_{t_n})^* a s^{(2)}_{t_{n+1}} y_{n+1}^*\\
& \approx_{\frac{\epsilon}{5}} & 
y_n (s^{(1)}_{t_n})^*  s^{(2)}_{t_{n+1}} a  y_{n+1}^* 
\makebox{ (by Definition \ref{df:PIExt} (2))}\\  
& \approx_{\frac{\epsilon}{5}} &
y_n (s^{(1)}_{t_n})^*  s^{(2)}_{t_{n}} a  y_{n+1}^* 
\makebox{ (by (\ref{equ:Mar620268AM}))}\\
& \approx_{\frac{\epsilon}{5}} & 0 \makebox{ (by Definition \ref{df:PIExt} (2)).} 
\end{eqnarray*}
Hence, for all sufficiently large $n$, with $n$ being odd,  
$x_n a x_{n+1}^* \approx_{\epsilon} 0$.
By a similar argument, for all sufficiently large $n$, with $n$ being even, 
$x_n a x_{n+1}^* \approx_{\epsilon} 0$.
Hence, since $\epsilon$ and $a$ are arbitrary, for all $a \in \A$,
$x_n a x_{n+1}^* \rightarrow 0$.

Again, let $\epsilon > 0$ and $a \in \A$ be arbitrary. 
For all sufficiently large $n$, with $n$ being odd, we have that
\begin{eqnarray*}
x_n a x_{n}^* & = & y_n (s^{(1)}_{t_n})^* a s^{(1)}_{t_{n}} y_n^*\\
& \approx_{\frac{\epsilon}{5}} &
y_n (s^{(1)}_{t_n})^*  s^{(1)}_{t_{n}} a  y_n^*
\makebox{ (by Definition \ref{df:PIExt} (2))}\\
& \approx_{\frac{\epsilon}{5}} & 
y_n a y_n^* \makebox{ (by Definition \ref{df:PIExt} (2))}\\
&  \approx_{\frac{\epsilon}{5}} & \psi(a) 
\makebox{  (by (\ref{equ:Mar620268PM})).}
\end{eqnarray*}
Hence, for all sufficiently large $n$, with $n$ being odd,
$x_n a x_{n}^* \approx_{\epsilon} \psi(a)$.
By a similar argument, for all sufficiently large $n$, with $n$ being even,
$x_n a x_{n}^* \approx_{\epsilon} \psi(a)$.
Since $\epsilon$ and $a$ are arbitrary, for all $a \in \A$,
$x_n a x_n^* \rightarrow \psi(a)$.  
\end{proof}

The following generalizes \cite{VoiculescuWvN} Lemma 1.2 (see also \cite{ArvesonExt} Section 2, Theorem 4
and its Corollary).   

\begin{thm} \label{thm:NgWangWvN1} 
Let $\B$ be a nonunital $\sigma$-unital C*-algebra, and let
$\A \subseteq \C(\B)$ be a separable nuclear C*-subalgebra, such that
the inclusion map $\A \hookrightarrow \C(\B)$ is strongly $O_{\infty}$-stable.

Suppose that $\psi : \A \rightarrow \C(\B)$ is a completely positive contractive
linear map for which $\psi(a) \in Ideal_{\C(\B)}(a)$ for all $a \in A$.

Then there exists an element $x \in \C(\B)$ such that 
$\psi(a) = x a x^*$ for all $a \in \A$.   
\end{thm}

\begin{proof}
The proof is a (easier) variation on the argument in \cite{KucerovskyNgCPMaps}
Theorem 3.4.  We provide the proof for the convenience of the reader.  

Since $\A$ is separable, let $\{ a_n \}_{n=1}^{\infty}$ be a dense 
sequence 
in the closed unit sphere of $\A_+$ (i.e., set of norm one elements of $\A_+$).
By Lemma \ref{lem:Mar720261AM}, let $\{ x_n \}$ be a (norm-) bounded sequence
in $\C(\B)$ such that for all $a \in \A$,
\begin{equation} \label{equ:Mar720262AM}
 x_n a x_{n+1}^* \rightarrow 0 \makebox{ and } 
x_n a x_n^* \rightarrow \psi(a). \end{equation} 
Let $M > 0$ be such that $\| x_n \| < M$ for all $n$. 

Since $\A$ is nuclear, by 
\cite{ChoiEffros}, let $\widetilde{\psi} : \A \rightarrow \Mul(\B)$ be 
a completely positive contractive linear map such that 
$\pi \circ \widetilde{\psi} = \psi$, where $\pi : \Mul(\B) \rightarrow 
\C(\B)$ is the quotient map.  

Let $\{ \widetilde{x}_n \}$ and $\{ \widetilde{a}_n \}$ be sequences in  
$\Mul(\B)$ such that  
$$\| \widetilde{x}_n \| \leq M, \makebox{ } \widetilde{a}_n \in \Mul(\B)_+,
\makebox{ } \| \widetilde{a}_n \| = 1, \makebox{ }
\pi(\widetilde{x}_n) = x_n \makebox{ and } \pi(\widetilde{a}_n) = a_n,  
\makebox{ for all }n.$$ 
Hence, by (\ref{equ:Mar720262AM}), 
passing to a subsequence of $\{ x_n \}$ if necessary, we may
assume that for all $1 \leq k \leq n$, we can find
$b_{n,k}, b'_{n,k} \in \B$ such that 
\begin{equation} \label{equ:Mar720263AM}
\widetilde{\psi}(a_k) \approx_{\frac{1}{2^{n}}}  \widetilde{x}_n 
\widetilde{a}_k \widetilde{x}_n^* + b_{n,k} \makebox{ and } 
\widetilde{x}_n   
\widetilde{a}_k \widetilde{x}_{n+1}^* \approx_{\frac{1}{2^{n}}}  b'_{n,k}.
\end{equation}

Since $\B$ is $\sigma$-unital, let $\{ e_n \}$ be a sequential approximate 
unit for $\B$, with $e_{n+1} e_n = e_n$ for all $n$, such that
$\{ e_n \}$ quasicentralizes the set 
$\{ \widetilde{x}_n, \widetilde{a}_n, \widetilde{\psi}_n(a_l) : n, l \geq 1 \}$.
Passing to a subsequence of $\{ e_n \}$ if necessary, we may assume that
for all $1 \leq l \leq  k \leq n$, the following statements are true:
\begin{eqnarray}
\label{Friesen1} \widetilde{\psi}(a_k) (e_{n} - e_{n-1})^{1/4} 
& \approx_{\frac{1}{2^{n}}} &
(e_{n} - e_{n-1})^{1/4} \widetilde{\psi}(a_k).\\  
\label{Friesen2} \widetilde{x}_k (e_{n} - e_{n-1})^{1/4}
& \approx_{\frac{1}{2^{n}}} &
(e_{n} - e_{n-1})^{1/4}\widetilde{x}_k.\\
\label{Friesen3} \widetilde{a}_k (e_{n} - e_{n-1})^{1/4} 
& \approx_{\frac{1}{2^{n}}} &
(e_{n} - e_{n-1})^{1/4} \widetilde{a}_k\\
\label{Friesen4} b_{k,l} (e_{n} - e_{n-1})^{1/4} 
& \approx_{\frac{1}{2^{n}}} &
0 \makebox{  } \approx_{\frac{1}{2^{n}}} (e_{n} - e_{n-1})^{1/4} b_{k,l}\\ 
\label{Friesen5} 
b'_{k,l} (e_{n} - e_{n-1})^{1/4} & \approx_{\frac{1}{2^{n}}} &
0 \makebox{  } \approx_{\frac{1}{2^{n}}} (e_{n} - e_{n-1})^{1/4} b'_{k,l}
\end{eqnarray} 

Before going on, we fix a terminology:  For any $X, Y 
\in \Mul(\B)$, 
we write $X = Y \makebox{ (mod $\B$)}$ to mean that $X - Y \in \B$.

Next, for all $n \geq 1$, let $y_n =_{df} (e_n - e_{n-1})^{1/4} 
\widetilde{x}_n 
(e_n - e_{n-1})^{1/4} \in \B$, and let $y \in \Mul(\B)$ be given by
$y =_{df} \sum_{n=1}^{\infty} y_n$, where the series converges in 
the strict topology.   
Then for all $k \geq 1$,  
\begin{eqnarray*}
y \widetilde{a}_k y^* & = & \sum_n y \widetilde{a}_k y_n^*\\  
& = & \sum_n y (e_n - e_{n-1})^{1/4} \widetilde{a}_k 
\widetilde{x}_n^* (e_n - e_{n-1})^{1/4}
\makebox{ (mod $\B$)}\\
& &   \makebox{(by definition of $y_n$ and 
(\ref{Friesen3}))}\\  
& = & \sum_n (y_{n-1} + y_n + y_{n+1}) (e_n - e_{n-1})^{1/4} 
\widetilde{a}_k \widetilde{x}_n^* (e_n - e_{n-1})^{1/4}\\
& & \makebox{(by definitions of $y, y_m$ and since $e_m e_{m-1}
= e_{m-1}$ for all $m$)}\\  
& = & \sum_n (y_{n-1} + y_n + y_{n+1}) \widetilde{a}_k y_n^* 
\makebox{ (mod $\B$)}\\
& & \makebox{(by (\ref{Friesen3}) and definition of $y_m$)}\\ 
\end{eqnarray*}

Now for all $k$, by the definition of $y_m$, (\ref{Friesen2}), and
(\ref{equ:Mar720263AM}) and (\ref{Friesen5}),  
$$\sum_n y_{n-1} \widetilde{a}_k y_n^* 
= \sum_n (e_{n-1} - e_{n-2})^{1/2} \widetilde{x}_{n-1} \widetilde{a}_k 
\widetilde{x}_n^* (e_n - e_{n-1})^{1/2} \makebox{ (mod $\B$) }  
= 0 \makebox{ (mod $\B$)}$$ 
So $\sum_n y_{n-1} \widetilde{a}_k y_n^* \in \B$. By a similar argument,
for all $k$, $\sum_n y_{n} \widetilde{a}_k y^*_{n-1} \in \B$.  
Plugging this into the computation of $y \widetilde{a}_k y^*$ 
above
(for all $k \geq 1$),
we have that for all $k \geq 1$, 
\begin{eqnarray*}
y \widetilde{a}_k y^* & = & \sum_n y_n \widetilde{a}_k y_n^* 
\makebox{  (mod $\B$)}\\
& = & \sum_n (e_n - e_{n-1})^{1/2} \widetilde{x}_n \widetilde{a}_k \widetilde{x}_n^* (e_n - e_{n-1})^{1/2}
\makebox{ (mod $\B$)}\\
& & \makebox{(by definition of $y_m$ and (\ref{Friesen2}))}\\
& = & \sum_n (e_n - e_{n-1})^{1/2} \widetilde{\psi}(a_k) (e_n - e_{n-1})^{1/2}
\makebox{ (mod $\B$)}\\
& & \makebox{(by (\ref{equ:Mar720263AM}) and (\ref{Friesen4}))}\\
& = & \widetilde{\psi}(a_k) \sum_n (e_n - e_{n-1})   \makebox{ (mod $\B$)}
\makebox{    (by (\ref{Friesen1}))}\\
& = & \widetilde{\psi}(a_k)
 \makebox{(mod $\B$)    (since $\{ e_n \}$ is an approximate unit for $\B$).}\\  
\end{eqnarray*}

Hence, if we let $x =_{df} \pi(y) \in \C(\B)$, then for all $k \geq 1$, 
$x a_k x^* = \psi(a_k)$.  
Since $\{ a_k \}$ is dense in the closed unit sphere of $\A_+$, we have that
$x a x^* = \psi(a)$ for all $a \in \A$. 
\end{proof}

\section{Classifying extensions}

We again introduce some notation.

\begin{para} \label{para:UEAndBDFSum}   
Let $\A$ and $\B$ be C*-algebras with $\B$ nonunital, and let
$\phi, \psi : \A \rightarrow \C(\B)$ be extensions (see \ref{FirstNotation}).
Recall that $\phi$ and $\psi$ are said to be \emph{weakly unitarily equivalent}
if there exists a unitary $u \in \C(\B)$ such that 
$u \phi(\cdot) u^* = \psi(\cdot)$.  
If, in addition, we can find a unitary $w \in \Mul(\B)$ for which $u = \pi(w)$,
then we say that $\phi$ and $\psi$ are \emph{unitarily equivalent}, 
and write $\phi \sim \psi$. We also let $[\phi]$ denote the 
unitary equivalence
class of $\phi$.

Suppose, in addition, that $\C(\B)$ is properly infinite.  The 
\emph{(generalized) Brown--Douglas--Fillmore sum} (or \emph{BDF sum})
of $\phi$ and $\psi$ is the extension $(\phi \oplus \psi)(\cdot) =_{df}
S \phi(\cdot) S^* + 
T \psi(\cdot) T^*$, where $S, T \in \C(\B)$ are isometries for which
$SS^* + TT^* \leq 1$.  
If $\phi$ and $\psi$ have large complement, then the BDF sum is well-defined
up to weak unitary equivalence.
If, in addition, $\C(\B)$ is also $K_1$-injective and 
$\phi$ and $\psi$ have large complement, then the BDF sum is well-defined
up to unitary equivalence.  (See \cite{GabeLinNg} Proposition 3.7 or
\cite{ChandNgSutradhar} Lemma 2.2.)   
\end{para}

We now define the object that we want to compute in our classification of extensions.
Of course, it may be empty until we impose more
conditions.        

\begin{df} \label{df:ExtSigma}
Let $\A$ be a separable C*-algebra,
and $\B$ be a $\sigma$-unital, nonunital C*-algebra.
Suppose that
$\sigma : \A \rightarrow \C(\B)$ is an essential extension.

Then $\ee_{\sigma}(\A, \B)$ consists of all $[\phi]$ 
(see \ref{para:UEAndBDFSum}) such that
\begin{enumerate}
\item[(a)] $\phi : \A \rightarrow \C(\B)$ is essential and has large complement,
\item[(b)] $\phi$ is strongly $O_{\infty}$-stable, and
\item[(c)]  $Ideal_{\C(\B)}(\phi(a)) = Ideal_{\C(\B)}(\sigma(a))$ for all
                $a \in \A$.
\end{enumerate}
\end{df}
\begin{para}
If, in the Definition \ref{df:ExtSigma}, 
$\sigma$ is also strongly $O_{\infty}$-stable and
has large complement, then $\ee_{\sigma}(\A, \B)$ is certainly nonempty.
\end{para}

\begin{para} \label{para:ExtBDFSum}
 Suppose, in addition, that $\C(\B)$ is properly infinite and $K_1$-injective.
Then by \ref{para:UEAndBDFSum}, the BDF sum induces a well-defined 
addition on $\ee_{\sigma}(\A, \B)$ (which we also call the \emph{BDF sum}) given by $[\phi] + [\psi] 
=_{df} [\phi \oplus \psi]$, for all $[\phi], [\psi]
\in \ee_{\sigma}(\A, \B)$. 
Thus, $\ee_{\sigma}(\A, \B)$ becomes an 
abelian semigroup.         
\end{para}

\begin{para}
For our classification of extensions, we use Kirchberg's interesting
ideal related KK (\cite{KirchbergIdealKK}), which generalizes Kasparov's KK.   A good reference for ideal related KK theory and 
the relevant background information (e.g, actions of topological spaces on C*-algebras) can be found in 
\cite{GabeOinfty}.  A short quick reference can be found in \cite{NgWangExt1}. \emph{For the rest of this paper, we will
 freely use the notation and terminology of \cite{NgWangExt1} without introduction.} 
\end{para}

 \begin{para}   \label{para:Actions} We next fix some notation to be used in the main theorem.
        Let $\D$ and $\E$ be C*-algebras, and
        let $\phi : \D \rightarrow \E$ be a  *-homomorphism.
        Let $X$ be the topological space given by $X =_{df} Prim(\E)$, i.e.,
        $X$ is the primitive ideal space of $\E$, endowed with the Jacobson or
        hull--kernel topology
        (see \cite{DavidsonBook}
        section VII.3 or \cite{PedersenBook} subsection 4.1.2).

        Let $\phi : \D \rightarrow \E$ be a  *-homomorphism,
         and
        let $\Phi_{\D} : \oo(X) \rightarrow \I(\D)$ and $\Phi_{\E} : \oo(X) \rightarrow
        \I(\E)$ be the
actions of $X$, on $\D$ and $\E$ respectively (see \cite{NgWangExt1} 4.11, 
4.10;  
 \cite{GabeOinfty} 10.7), that are given as
        follows:

        \begin{equation} \label{equ:Oct3020241AM}
                \Phi_{\E} : G \mapsto \bigcap_{\J \in X - G} \J \makebox{ and }
                \Phi_{\D} : G \mapsto \phi^{-1}(\Phi_{\E}(G)).
        \end{equation}
\end{para}

Here is our (BDF-type) classification of extensions.  As far as we know, it contains
all previous (up to (strong) unitary equivalence) classifications of essential
extensions (with large complement, and of $\sigma$-unital C*-algebras 
 by separable nuclear C*-algebras).  

\begin{thm} \label{thm:MainThEnd}
        Let $\A$ be a separable nuclear C*-algebra, and let $\B$ be a
        nonunital, $\sigma$-unital C*-algebra 
such that $\C(\B)$ is $K_1$-injective and properly infinite. 

        Let $X =_{df} Prim(\C(\B))$ (i.e., $X$ is the primitive
ideal space of $\C(\B)$,  given
the Jacobson topology), and let $\sigma : \A \rightarrow
        \C(\B)$ be a strongly $O_{\infty}$-stable, essential
extension with large complement.  
Let
$\Phi_{\A} : O(X) \rightarrow Ideal(\A)$ and
        $\Phi_{\C(\B)} : O(X) \rightarrow Ideal(\C(\B))$ be the actions on
        $\A$ and $\C(\B)$ respectively, which are defined as in
        (\ref{equ:Oct3020241AM}) (with $\D$, $\E$, $\phi$ replaced with
        $\A$, $\C(\B)$, $\sigma$ respectively).

        Then we have the following:
\begin{enumerate}
\item $\ee_{\sigma}(\A, \B)$ is an abelian group.
\item The map
        $$\Gamma : \ee_{\sigma}(\A, \B) \rightarrow KK(X; \A, \C(\B)) :
        [\psi]_{Ext} \mapsto [\psi]_{KKX}$$
        is a group isomorphism.    
\end{enumerate}
        (See Definition \ref{df:ExtSigma}, \ref{para:ExtBDFSum} and
        \cite{NgWangExt1} 4.14; \cite{GabeOinfty} Ch. 12.)
\end{thm}

\begin{proof}
This argument is a (much easier) variation on that of \cite{NgWangExt1} Theorem 4.15.
We provide the argument for the convenience of the reader.

    Note that by the definition of $\ee_{\sigma}(\A, \B)$ (see
        Definition \ref{df:ExtSigma}) and by \cite{NgWangExt1}
 Lemma 4.13
        item (5),
        if $[\psi]_{Ext} \in
        \ee_{\sigma}(\A, \B)$, then $\psi: \A \rightarrow \C(\B)$ is
        $X$-equivariant, and hence, $\psi$ induces an element
        $[\psi]_{KKX} \in KK(X; \A, \C(\B))$. (See the end of the last paragraph
        of 4.14 in \cite{NgWangExt1}.)
        Hence, that $\Gamma$ is a well-defined map follows from the
        above and from
        the basic properties of ideal related KK theory (see, for example,
        \cite{NgWangExt1} 4.14 and \cite{GabeOinfty} Section 12).

        We next prove the surjectivity of the map $\Gamma$.
        Our first goal is to  check that the hypotheses of the existence theorem 
        of \cite{GabeOinfty} Proposition 15.6 are satisfied.
        Firstly,  $\A$ is a separable nuclear C*-algebra. Also,  by
        \cite{NgWangExt1} Lemma 4.13 item (2),
         $\A$ is a lower semicontinuous $X$-C*-algebra in the sense of
        \cite{GabeOinfty} Definition 10.8.

        Next, since $\C(\B)$ is properly infinite, it contains a unital copy of
        $O_{\infty}$, and hence, we have a sequence $\{ p_n \}_{n=1}^{\infty}$
        of pairwise orthogonal projections in $\C(\B)$ such that for all $n$,
        $p_n \sim 1_{\C(\B)}$.  Let $d \in \C(\B)$ be given by
        $d =_{df} \sum_{n=1}^{\infty} \frac{p_n}{n}$.  Then
        $\D_1 =_{df} \overline{d \C(\B) d}$ is a $\sigma$-unital full hereditary
        C*-subalgebra of $\C(\B)$.  Moreover,  by the definition of $d$
and since $d$ is a strictly positive element of $\D_1$,
 it is not hard to check that
        $\D_1$ satisfies the
        characterization of stability in \cite{HjelmborgRordam} Theorem 3.3.
Hence, by
        \cite{HjelmborgRordam} Theorem 3.3, $\D_1$ is stable.  Hence, $\D_1$ is a $\sigma$-unital,
        full, stable, hereditary
        C*-subalgebra of $\C(\B)$.

        Next, by \cite{NgWangExt1} Lemma 4.13 item (4),
$\sigma : \A \rightarrow
        \C(\B)$ is $X$-equivariant and  $X$-full.
Moreover, by hypothesis, $\sigma$ is a  strongly
        $O_{\infty}$-stable *-homomorphism.

        From the above, the hypotheses of \cite{GabeOinfty} Proposition 15.6 are
        satisfied.  Now let $\alpha \in KK(X; \A, \C(\B))$ be arbitrary.
        By \cite{GabeOinfty} Proposition 15.6,
        there exists an $X$-equivariant,  $X$-full and strongly $O_{\infty}$-stable
 *-homomorphism
        $\phi : \A \rightarrow \C(\B)$ such that
        $[\phi]_{KKX} = \alpha$.
        Replacing $\phi$ with $\phi \oplus 0$
(where $\oplus$ is the BDF sum; see the last paragraph of
\ref{para:UEAndBDFSum})  if
        necessary, we may assume that $\phi$ has large complement.
        Moreover, by \cite{NgWangExt1} Lemma 4.13 item (5),
        $Ideal(\phi(a)) = Ideal(\sigma(a))$ for all $a \in \A$ (which implies
        that $\phi$ is injective, since $\sigma$ is injective -- i.e.,
        $\phi$ is an essential extension).
        Thus, $\phi$ induces an element
        $[\phi]_{Ext} \in \ee_{\sigma}(\A, \B)$ (see Definition
        \ref{df:ExtSigma}) and $\Gamma([\phi]_{Ext})  = \alpha$.
        Since $\alpha \in KK(X; \A, \B)$ is arbitrary, $\Gamma$ is surjective.

We now prove that $\Gamma$ is injective.
        Again, let us firstly the check that the hypotheses of the uniqueness
        theorem \cite{GabeOinfty} Theorem F are satisfied.
        Firstly, note that
        $\A$ is a separable nuclear C*-algebra.  Also, by \cite{NgWangExt1} Lemma
        4.13 item (2), $\A$ is a lower semicontinuous
        $X$-C*-algebra in the sense of \cite{GabeOinfty} 10.8.

        Now let $[\phi]_{Ext}, [\psi]_{Ext} \in \ee_{\sigma}(\A, \B)$
be such that
$\Gamma([\phi]_{Ext}) = \Gamma([\psi]_{Ext})$; i.e.,
 $[\phi]_{KKX} = [\psi]_{KKX}$ in $KK(X; \A, \C(\B))$.
So $\phi$ and $\psi$ both are strongly $O_{\infty}$-stable,  
essential extensions
with large complement and
        $Ideal(\phi(a)) = Ideal(\sigma(a)) = Ideal(\psi(a))$
for all $a \in \A$.
Hence, by \cite{NgWangExt1} Lemma 4.13 item (5), $\phi$ and $\psi$ are both $X$-equivariant and $X$-full.

Hence, all the hypotheses of \cite{GabeOinfty} Theorem F are fullfilled.
Hence, by \cite{GabeOinfty} Theorem F, $\phi$ and $\psi$ are asymptotically
Murray--von Neumann equivalent in the sense of \cite{GabeOinfty} (1.6).
Hence, since $\C(\B)$ is properly infinite and $\phi$ and $\psi$ both have large complement, 
and by \cite{GabeO2}   
Proposition 3.10, $\phi$ and $\psi$ are asymptotically unitarily equivalent.
Hence, by \cite{GabeLinNg} Lemma 4.3 (see also \cite{NgWangExt1} Lemma 3.7),
$\phi$ and $\psi$ are weakly unitarily equivalent. 
Hence, by \cite{GabeLinNg} Proposition 3.7 (see also
\cite{ChandNgSutradhar} Lemma 2.2 for another proof), since
$\C(\B)$ is $K_1$-injective,  it follows that
$\phi$ and $\psi$ are unitarily equivalent.  Hence, $[\phi]_{Ext} =
[\psi]_{Ext}$.  Since $[\phi]_{Ext}$ and $[\psi]_{Ext}$ 
are arbitrary,    $\Gamma$
is injective.   

Now it is clear that $\Gamma$ respects addition (e.g., see \ref{para:ExtBDFSum} and \cite{GabeOinfty} section 12.3).
Thus, $\Gamma$ is a morphism of abelian semigroups.
Hence, since $KK(X; \A, \B)$ is an abelian group, and since $\Gamma$ is bijective, $\ee_{\sigma}(\A, \B)$ is an abelian
group and $\Gamma$ is a group isomorphism.
 
\end{proof}
 
\begin{para}
Suppose that $\B$ is a nonunital $\sigma$-unital C*-algebra such that either i. $\B$ is stable, ii. $\B$ is simple with continuous scale,  iii.  
$\B$ is simple, pure, 
with quasicontinuous scale, or iv. $\B \cong \B \otimes O_{\infty}$. 
Then $\C(\B)$ is $K_1$-injective.  (E.g., for i., ii., and iii.,  
see \cite{ChandNgSutradhar} Theorem 7.9 items 7, 5 and 6;  the proofs do not require that $\B$ be separable.  For i., simplicity is also not 
required.  Sketch of the proof for iv.:  By the proof of Proposition \ref{prop:Examples1}, $\Mul(\B)$ is separably $O_{\infty}$-stable.  Hence,
$\C(\B)$ is separably $O_{\infty}$-stable.   Hence,
by \cite{RohdeThesis} Theorem 6.2.3,  $\C(\B)$ is $K_1$-injective.)
So the $K_1$-injectivity condition of Theorem \ref{thm:MainThEnd} does indeed generalize an implicit condition in
previous classifications.      Moreover, it is an open question, of Blanchard, Rohde and Rordam,
whether 
every unital properly infinite C*-algebra is $K_1$-injective (\cite{BlRohRor}).
\end{para}

\end{document}